\documentclass[12pt, titlepage]{amsart}

\usepackage{latexsym, color}
\usepackage[numbers]{natbib}
\usepackage{graphicx}

\usepackage{amsmath, fullpage, booktabs}
\usepackage[colorlinks=TRUE, citecolor=blue]{hyperref}

\allowdisplaybreaks

\usepackage{fullpage}
\usepackage{setspace}
\usepackage[nolists, tablesfirst]{endfloat}

\newtheorem{Theorem}{Theorem}

\newcommand{\dif}{\mathrm{d}}

\begin{document}

\title{On Occupation Time for On-Off Processes with Multiple Off-States}

\author[C. Hu, V. Pozdnyakov and J. Yan]
{Chaoran Hu$^{1*}$, Vladimir Pozdnyakov$^{1}$,  and Jun Yan$^{1}$}

\thanks{\\
1. Department of Statistics, University of Connecticut, 215 Glenbrook Road, Storrs, CT 06269-4120\\
* Corresponding author. E-mail: {\tt chaoran.hu@uconn.edu}
}

\maketitle

\begin{abstract}
% SHORT version
% A Markov renewal on-off process with multiple off-states is introduced. The
% exact joint and marginal distributions of the off-state occupation times are derived. A
% special case when all holding times have L\'evy distributions is considered.

The need to model a Markov renewal on-off process with multiple off-states
arise in many applications such as economics, physics, and engineering.
Characterization of the occupation time of one specific off-state
marginally or two off-states jointly is crucial to understanding such
processes. We derive the exact marginal and joint distributions of the
off-state occupation times. The theoretical results are confirmed
numerically in a simulation study. A special case when all holding
times have L\'evy distribution is considered for the possibility of
simplification of the formulas.


% \medskip
% \noindent{\sc 2010 Mathematics Subject Classification:} 60J27 (62M05)
%\subclass{60J27 \and 62M05}

\end{abstract}

% \maketitle

\section{Introduction}

Markov renewal on-off processes, also known as alternating renewal
processes or telegraph processes, arise in applications in a variety of fields.
Classic alternating renewal processes were first studied by
\citet{Cane:beha:1959}, \citet{Page:theo:1960} and \citet{Newman:1968}
with appxlications to animal ethology, maintenance of
electronic equipment, and communication engineering, respectively.
Recently, the telegraph process and their variations are employed in
various fields of applications such as pricing in mathematical finance
and insurance
\citep[e.g.,][]{Kolesnik:Ratanov:2013, DiCrescenzo:etal:2013,
  DiCrescenzo:etal:2014, DiCrescenzo:Zhacs:2015, Ratanov:2017},
modelling the propagation of a damped wave in physics
\citep[e.g.,][]{DeGregorio:etal:2012, Macci:2016}, and inventory and storage
models in engineering \citep{Xu:etal:2015}. In certain applications,
the off-state of an on-off process has multiple types, which raises
new questions about its properties such as the distribution of the
occupation time in specific off-state type marginally or multiple
off-state types jointly. Our interest in on-off processes with
multiple off-states
is motivated by our recent work in animal movement modeling,
where a predator can have different non-moving states such as
resting or handling a kill \citep{Pozd:etal:2017, Pozd:etal:2018}.
The occupation time in the handling state is important for ecologists
to understand the behavior of predators.

To fix ideas, consider a server that has two different types of
failures, each requiring
a different time and cost to repair. One basic question is: if the server
is on at time~0, what is the distribution of the time one spends fixing the
type one failures by time $t>0$?  We model the server state process with the
following Markov renewal process.
The process starts in state~0 (the on-state) and spend there a random holding
time according to a given absolutely continuous distribution. When the first holding
time is over, we flip an asymmetric coin to decide which type of failure (off-state 1
or off-state 2) comes next. The holding times in the off-states are also
absolutely continuous, generally with different distributions. Once the second holding
time is over the process returns to the on-state (state 0) and then we repeat the
construction process.

The occupation time of a specific type of the off-state is our focus.
For the on-state, the distribution of the occupation time is known from the
results on the telegraph process
\citep{Perr:Stad:Zack:firs:1999, DiCrescenzo:2001,
  Stad:Zack:tele:2004,  Zack:gene:2004}.
Indeed, if we collapse the two off-states into one, then the resulting
process is a regular on-off process (or alternating renewal process)
whose off-state holding time distribution is a mixture of the two original off-state
holding time distributions. If the different types of repair require different
resources, however, we want to know the distribution of the occupation time in
a particular type of off-state. For that task,  the results on classical telegraph
processes are not directly applicable.
Moreover, assume that the cost for being
in an off-state is proportional to its occupation time and different for different
off-states. If one wants to know the distribution of the total cost for repairs
of all kinds by a fixed time, then we need the joint distribution of
the two types of off-state occupation times.

Our problem is related to but different from some recent works.
An extension of the telegraph process to a process with three states
is studied in \citet{Bshouty:etal:2012}, where within a
renewal cycle all three states are visited in a given deterministic order.
In our case, however, we have an on-off process with two off-states;
only two states (the on-state and an off-state) are visited in each renewal
cycle, but the off-states is chosen randomly.
Another related work is \citet{Crimaldi:etal:2013}, where
there are two states, but at each random epoch, the new state is
determined by the outcome of a random trial, and, as a result, the
process can stay in the same state.

The key idea to study the occupation time of a specific off-state is to
exploit a certain periodicity of our Markov renewal process.
In fact, when analyzing any Markov renewal process, it is always convenient to
do the conditioning on returning to a certain state. For the on-off process
with two off-states, state~0 has an additional nice property.
The numbers of steps between two consecutive visits of
state~0 is not {\it random}, and it is always equal to~2.
We derive the marginal distribution of the occupation time of a
specific off-state first, and then we modify our derivation to obtain
the joint distribution of the two off-state occupation times.
Of course, one can get the marginal distribution by integrating the
joint one. Our approach, however, is easier to follow.

\section{Algorithmic Construction of On-Off Process with Two Off-States}

Suppose that we are given the following collection of independent sequences of
non-negative random variables:
\begin{enumerate}
\item $\{U_k\}_{k\geq 1}$ are independent identically distributed (iid) random
  variables with cumulative distribution function (cdf) $F_U$ and probability
  density function (pdf) $f_U$ (these random variables will be used
  as the holding times when the server is up, state~0);
\item $\{S_k\}_{k\geq 1}$ are iid random variables with cdf $F_S$ and pdf $f_S$
  (holding times when the server is down for short repairs, state~1);
\item $\{L_k\}_{k\geq 1}$ are iid random variables with cdf $F_L$ and pdf $f_L$
  (holding times when the server is down for long repairs, state~2);
\item $\{\xi_k\}_{k\geq 1}$ are iid random variables with $\Pr(\xi_k=1)=p_1>0$
  and $\Pr(\xi_k=0)=p_2=1-p_1$.
\end{enumerate}

Now, we present our construction of the on-off process with two off-states, $X(t)$.
\begin{enumerate}
\item Initialize with $X(0)=0$ and $T_0=0$.
\item For cycles $i = 1, 2, \ldots$:
  \begin{enumerate}
  \item
    Let $T_{2i - 1}=U_{i}+T_{2i-2}$, and $X(t)=0$ for all
    $t\in[T_{2i-2},  T_{2i - 1})$.
  \item
    If $\xi_i=1$ then $T_{2i}=T_{2i - 1} + S_i$, and $X(t)=1$ for all
    $t\in[T_{2i - 1}, T_{2i})$;
    otherwise, $T_{2i}=T_{2i - 1}+L_i$, and $X(t)=2$ for all $t\in[T_{2i - 1}, T_{2i})$.
  \end{enumerate}
\end{enumerate}

Then the occupation times in state~0, state~1, and state~2 are, respectively,
\begin{align*}\label{occupation*times}
U(t)&=\int_0^t1_{\{X(s)=0\}}\dif s,\\
S(t)&=\int_0^t1_{\{X(s)=1\}}\dif s,\\
L(t)&=\int_0^t1_{\{X(s)=2\}}\dif s.
\end{align*}
The corresponding defective marginal densities of $S(t)$ and $L(t)$ are denoted as
\begin{align*}
 p_{Sj}(s, t) &=  \Pr(S(t) \in \dif s,X(t)=j)/\dif s, \\
 p_{Lj}(s, t) &=  \Pr(L(t) \in \dif s,X(t)=j)/\dif s,
\end{align*}
where $t \geq 0$, $0<s<t$, $j = 0,1,2$.
The defective joint two-dimensional density of $S(t)$ and $L(t)$ for
$u,v>0$, $u+v<t$ is denoted as
\begin{equation*}
 p_{SLj}(u, v, t) =  \frac{1}{\dif u \dif v} \Pr(S(t) \in \dif u, L(t) \in \dif v, X(t)=j),
\end{equation*}
where $j = 0,1,2$.
The densities are defective in the following sense: since both
occupation times $S(t)$ and $L(t)$ have an atom at 0, all three sums
\begin{equation*}
\sum_{j=0}^2\int_0^t p_{Sj}(s, t)\dif s, \quad \sum_{j=0}^2\int_0^t p_{Lj}(s, t)\dif s,
\end{equation*}
and
\begin{equation*}
\sum_{j=0}^2\iint_{u,v>0,\, u+v<t} p_{SLj}(u,v,t)\dif u \dif v
\end{equation*}
are less than~1.

\section{Marginal Distribution of Occupation Time of an Off-State}

To derive the distribution of occupation time we will need some auxiliary
random variables. Let $N(t)$ be the number of cycles (or returns to the on-state)
by time $t$. Formally, for $n\geq 0$
\begin{equation*}\label{number*of&cycles}
N(t)=n\quad\quad\mbox{ iff }\quad\quad T_{2n}\leq t<T_{2n+2}.
\end{equation*}
Finally, let $D_k=\xi_k S_k+(1-\xi_k)L_k$. These random variables are associated
with the off-state holding time of the regular on-off process when two off-states
are combined.

Our formulas include convolutions of different distributions. We will use the
following notation. If we are given a cdf $G(\cdot)$ and its pdf $g(\cdot)$,
then $G^{(n)}(\cdot)$ denotes $n$-fold convolution of $G(\cdot)$, and
$g^{(n)}(\cdot)$ denotes the $n$-fold convolution of $g(\cdot)$. If we are
given two cdfs $G(\cdot)$ and $H(\cdot)$ with pdfs $g(\cdot)$ and $h(\cdot)$,
then $G*H(\cdot)$ denotes the convolution of cdfs $G(\cdot)$ and $H(\cdot)$,
and $g*h(\cdot)$ denotes the convolution of pdfs $g(\cdot)$ and $h(\cdot)$.

We also use the following conventions. Any summation over the empty set is~0.
Zero-fold convolution
$G^{(0)}(\cdot)$ is the cdf of a random variable that is equal to~0 with probability~1.
Finally, $g^{(k)}*h^{(0)}(\cdot)=g^{(k)}(\cdot)$ for any $k\geq 1$.

\begin{Theorem}\label{occupation*time}
Let $t\geq 0$ and $0<s<t$. Then
\begin{align*}
\Pr(S(t)=0,X(t)=0)&=\sum_{n=0}^\infty\left[F_U^{(n)}*F_L^{(n)}(t)-F_U^{(n+1)}*F_L^{(n)}(t)\right]p_2^n,\\
\Pr(S(t)=0,X(t)=1)&=0,\\
\Pr(S(t)=0,X(t)=2)&=\sum_{n=0}^\infty\left[F_U^{(n+1)}*F_L^{(n)}(t)-F_U^{(n+1)}*F_L^{(n+1)}(t)\right]p_2^{n+1},
\end{align*}
and
\begin{align*}
p_{S0}(s, t)&=\sum_{n=1}^\infty\sum_{k=1}^n\binom{n}{k}p_1^kp_2^{n-k}f_S^{(k)}(s)\left[F_U^{(n)}*F_L^{(n-k)}(t-s)-F_U^{(n+1)}*F_L^{(n-k)}(t-s)\right],\\
p_{S1}(s, t)&=\sum_{n=0}^\infty\sum_{k=0}^n \binom{n}{k}p_1^{k+1}p_2^{n-k}f_U^{(n+1)}*f_L^{(n-k)}(t-s)\left[F_S^{(k)}(s)-F_S^{(k+1)}(s)\right],\\
p_{S2}(s, t)&=\sum_{n=1}^\infty\sum_{k=1}^n\binom{n}{k}p_1^kp_2^{n-k+1}f_S^{(k)}(s)\left[F_U^{(n+1)}*F_L^{(n-k)}(t-s)-F_U^{(n+1)}*F_L^{(n-k+1)}(t-s)\right].
\end{align*}
\end{Theorem}
\begin{proof}
First, note that the distribution of $S(t)$ has an atom. Indeed, if $U_1>t$ or
all the failures that occur before $t$ are of the second type, then $S(t)=0$.
More specifically, by conditioning on the number of returns to the on-state,
$N(t)$, we obtain that
\begin{align*}
  &\phantom{=\;} \Pr(S(t)=0,X(t)=0)\\
  &=\sum_{n=0}^\infty \Pr\left(S(t)=0,X(t)=0, N(t)=n\right)\\
  &=\Pr(U_1>t)+\sum_{n=1}^\infty \Pr\left(\sum_{j=1}^n(U_j+L_j)\leq t,\sum_{j=1}^n(U_j+L_j)+U_{n+1}>t, \sum_{j=1}^n\xi_j=0\right)\\
  &=\Pr(U_1>t)+\sum_{n=1}^\infty \Pr\left(\sum_{j=1}^n(U_j+L_j)\leq t,\sum_{j=1}^n(U_j+L_j)+U_{n+1}>t\right)p_2^n\\
  &=\Pr(U_1>t)+\sum_{n=1}^\infty\left[\Pr\left(\sum_{j=1}^nU_j+\sum_{j=1}^n L_j\leq t\right)-\Pr\left(\sum_{j=1}^{n+1}U_j+ \sum_{j=1}^n L_j\leq t \right) \right]p_2^n\\
  &=(1-F_U(t))+\sum_{n=1}^\infty\left[F_U^{(n)}*F_L^{(n)}(t)-F_U^{(n+1)}*F_L^{(n)}(t)\right]p_2^n\\
  &=\sum_{n=0}^\infty\left[F_U^{(n)}*F_L^{(n)}(t)-F_U^{(n+1)}*F_L^{(n)}(t)\right]p_2^n.
\end{align*}

Next, again by conditioning on $N(t)$, we get that for $0<s<t$
\begin{equation*}
\Pr(S(t)\in \dif s,X(t)=0)=\sum_{n=1}^\infty \Pr\left(S(t)\in \dif s,X(t)=0, N(t)=n\right).
\end{equation*}
Note that the summation starts from 1, because $X(t)=0$ and $N(t)=0$ implies
that $U_1>t$, and, therefore, $S(t)=0$.

The next step is to fix the number of switches to failures of type~1. Since the
total numbers of switches is $n$ and there is at least one failure of type~1, we have
\begin{equation*}
\Pr(S(t)\in \dif s,X(t)=0, N(t)=n)=\sum_{k=1}^n \Pr\left(S(t)\in \dif s,X(t)=0, N(t)=n,\sum_{j=1}^n\xi_j=k\right).
\end{equation*}
Because
\begin{equation*}X(t)=0, N(t)=n \quad\quad\mbox{ iff }\quad\quad \sum_{j=1}^n(U_j+D_j)\leq t, \sum_{j=1}^n(U_j+D_j)+U_{n+1}>t,
\end{equation*}
and it really does not matter during which cycles the switches to failures of
type 1 occur, we find that
\begin{align*}
& \quad \Pr\Big(S(t)\in \dif s,X(t)=0, N(t)=n,\sum_{j=1}^n\xi_j=k\Big)\\
&=\Pr\left(S(t)\in \dif s,\sum_{j=1}^n(U_j+D_j)\leq t, \sum_{j=1}^n(U_j+D_j)+U_{n+1}>t,\sum_{j=1}^n\xi_j=k\right)\\
&=\binom{n}{k}\Pr\left(S(t)\in \dif s,\sum_{j=1}^n(U_j+D_j)\leq t, \sum_{j=1}^n(U_j+D_j)+U_{n+1}>t,\sum_{j=1}^k\xi_j=k,\sum_{j=k+1}^n\xi_j=0\right).
\end{align*}
Next, observe that in this case $S(t)=\sum_{j=1}^kS_j$. Using independence between
$\{\xi_k\}_{k\geq 1}$ and the holding time sequences we have
\begin{align*}
&\quad \Pr\Big(S(t)\in \dif s,X(t)=0, N(t)=0,\sum_{j=1}^n\xi_j=k\Big)\\
&=\binom{n}{k}p_1^kp_2^{n-k}\Pr\left(\sum_{j=1}^kS_j\in \dif s,\sum_{j=1}^nU_j+\sum_{j=1}^k S_j+ \sum_{j=k+1}^n L_j\leq t, \sum_{j=1}^{n+1}U_j+\sum_{j=1}^k S_j+ \sum_{j=k+1}^n L_j>t\right)\\
&=\binom{n}{k}p_1^kp_2^{n-k}\Pr\left(\sum_{j=1}^kS_j\in \dif s,\sum_{j=1}^nU_j+\sum_{j=k+1}^n L_j\leq t-s, \sum_{j=1}^{n+1}U_j+ \sum_{j=k+1}^n L_j>t-s\right).
\end{align*}
Finally, independence of holding time sequences gives us that
\begin{align*}
& \quad \Pr\Big(\sum_{j=1}^kS_j\in \dif s,\sum_{j=1}^nU_j+\sum_{j=k+1}^n L_j\leq t-s, \sum_{j=1}^{n+1}U_j+ \sum_{j=k+1}^n L_j>t-s\Big)\\
&=\Pr\left(\sum_{j=1}^kS_j\in \dif s\right)\left[\Pr\left(\sum_{j=1}^nU_j+\sum_{j=k+1}^n L_j\leq t-s\right)-\Pr\left(\sum_{j=1}^{n+1}U_j+ \sum_{j=k+1}^n L_j\leq t-s \right) \right]\\
&=f_S^{(k)}(s)\left[F_U^{(n)}*F_L^{(n-k)}(t-s)-F_U^{(n+1)}*F_L^{(n-k)}(t-s)\right]\dif s.
\end{align*}

Now let us consider the case when $X(t)=1$ (at time $t$ the server is
down for a short repair). One difference is that there is no atom in this case.
As before, for $0<s<t$ we have
\begin{align*}
  &\phantom{=\;} \Pr(S(t)\in \dif s,X(t)=1)\\
  &=\sum_{n=0}^\infty \Pr\left(S(t)\in \dif s,X(t)=1, N(t)=n\right)\\
  &=\sum_{n=0}^\infty\sum_{k=0}^n \Pr\left(S(t)\in \dif s,X(t)=1, N(t)=n,\sum_{j=1}^n\xi_j=k, \xi_{n+1}=1\right)\\
  &=\sum_{n=0}^\infty\sum_{k=0}^n \binom{n}{k}\Pr\left(S(t)\in \dif s,X(t)=1, N(t)=n,\sum_{j=1}^k\xi_j=k, \sum_{j=k+1}^n\xi_j=0, \xi_{n+1}=1\right).
\end{align*}
Note that event
\begin{equation*}
\Big\{S(t)\in \dif s,X(t)=1, N(t)=n,\sum_{j=1}^k\xi_j=k, \sum_{j=k+1}^n\xi_j=0, \xi_{n+1}=1\Big\}
\end{equation*}
can be rewritten as
\begin{align*}
\Big\{S(t)\in \dif s, \; &\sum_{j=1}^{n+1}U_j+\sum_{j=1}^k S_j+ \sum_{j=k+1}^n L_j\leq t, \\
                     &\sum_{j=1}^{n+1}U_j+\sum_{j=1}^k S_j+ \sum_{j=k+1}^n
                       L_j+S_{n+1}> t,
                       \; \sum_{j=1}^k\xi_j=k,
                       \; \sum_{j=k+1}^n\xi_j=0, \xi_{n+1}=1\Big\}.
\end{align*}
Since $S(t)=t-\sum_{j=1}^{n+1}U_j- \sum_{j=k+1}^n L_j$, we finally obtain
\begin{align*}
&\quad \Pr(S(t)\in  \dif s,X(t)=1)\\
&=\sum_{n=0}^\infty\sum_{k=0}^n \binom{n}{k}p_1^{k+1}p_2^{n-k}\Pr\left(\sum_{j=1}^{n+1}U_j+\sum_{j=k+1}^n L_j\in t- \dif s,\sum_{j=1}^k S_j\leq s,\sum_{j=1}^k S_j+S_{n+1}>s \right)\\
&=\sum_{n=0}^\infty\sum_{k=0}^n \binom{n}{k}p_1^{k+1}p_2^{n-k}f_U^{(n+1)}*f_L^{(n-k)}(t-s)\left[F_S^{(k)}(s)-F_S^{(k+1)}(s)\right]\dif s.
\end{align*}
The other formulas can be derived in a similar way.
\end{proof}

Theorem~\ref{occupation*time} gives us the distribution of occupation time in
state~1. Since both off-states enter our story in a completely symmetric way,
the formulas for the occupation time in state~2
can be obtained by interchanging state~1 and~2 in Theorem~\ref{occupation*time}.

\section{Joint Distribution of Off-State Occupation Times}

As mentioned in the introduction, if we are interested in the
distribution of the total cost, then
we need the joint distribution of off-state occupation times.
More specifically, assume that the total cost
$C(t)$ is a linear function of occupation times, that is,
\begin{equation*}
C(t) = C_0 U(t) + C_1 S(t) + C_2 L(t).
\end{equation*}
Then the distribution of $C(t)$ is fully determined by the joint
distribution of $S(t)$ and $L(t)$. Note
that we do not need the joint distribution of all three occupation
times, because $U(t)=t-S(t)-L(t)$.

Since both $S(t)$ and $L(t)$ have atoms at 0, for every value of $X(t)$ we have four cases:
(1) both occupation times are~0;
(2) $S(t)=0$, $L(t)>0$;
(3) $S(t)>0$, $L(t)=0$;
and (4) both $S(t)$ and $L(t)$ are strictly greater than 0.
In total, we have 12~formulas. Some of them are trivial.
For instance, event $\{S(t)=0, X(t)=1\}$ has
probability~0, therefore, the corresponding defective one-dimensional
density of $L(t)$ is also~0. Moreover,
since there is a certain symmetry between $S(t)$ and $L(t)$ some
formulas can be found by interchanging state~1 and~2.
That is why in the following theorem below we have only 5~formulas.

\begin{Theorem}\label{joint*occupation*time}
Let $u,v>0$ and $u+v<t$. Then
\begin{align*}
\Pr(S(t)=0,L(t)=0,X(t)=0)&=1-F_U(t),\\
\Pr(S(t)=0,L(t)\in \dif v,X(t)=0)/\dif v&=\sum_{n=1}^\infty f_L^{n}(v)\left[F_U^{(n)}(t-v)-F_U^{(n+1)}(t-v)\right]p_2^n,\\
\Pr(S(t)\in \dif u, L(t)=0, X(t)=1)/\dif u&=\sum_{n=0}^\infty p_1^{n+1}f_U^{(n+1)}(t-u)\left[F_S^{(n)}(u)-F_S^{(n+1)}(u)\right],
\end{align*}
and
\begin{align*}
p_{SL0}(u, v, t)&=\sum_{n=1}^\infty\sum_{k=1}^{n-1}\binom{n}{k}p_1^kp_2^{n-k}f_S^{(k)}(u)f_L^{(n-k)}(v)\left[F_U^{(n)}(t-u-v)-F_U^{(n+1)}(t-u-v)\right],\\
p_{SL1}(u, v, t)&=\sum_{n=0}^\infty\sum_{k=0}^{n-1} \binom{n}{k}p_1^{k+1}p_2^{n-k}f_U^{(n+1)}(t-u-v)f_L^{(n-k)}(v)\left[F_S^{(k)}(u)-F_S^{(k+1)}(u)\right].
\end{align*}
\end{Theorem}

\begin{proof}
We will only derive the formula for $p_{SL1}(u, v, t)$. The remaining formulas can be obtained by similar modifications of the proof of Theorem~\ref{occupation*time}.

As before, we start with partitioning with respect to events $\{N(t)=n\}$.
More specifically, for $u,v>0$ and $0<u+v<t$ we have
\begin{align*}
  & \phantom{=\;} p_{SL1}(u, v, t)\dif u \dif v \\
  &=\sum_{n=0}^\infty \Pr\left(S(t)\in \dif u, L(t)\in \dif v, X(t)=1, N(t)=n\right)\\
  &=\sum_{n=0}^\infty\sum_{k=0}^{n-1} \Pr\left(S(t)\in \dif u, L(t)\in \dif v, X(t)=1, N(t)=n,\sum_{j=1}^n\xi_j=k, \xi_{n+1}=1\right)\\
  &=\sum_{n=0}^\infty\sum_{k=0}^{n-1} \binom{n}{k}\Pr\left(S(t)\in \dif u, L(t)\in \dif v,  X(t)=1, N(t)=n,\sum_{j=1}^k\xi_j=k, \sum_{j=k+1}^n\xi_j=0, \xi_{n+1}=1\right).
\end{align*}
Note that  now the upper limit of the inner summation is $n-1$, because $\sum_{j=1}^n\xi_j=n$ and $\xi_{n+1}=1$ implies that $L(t)=0$.

Next, observe that event
\begin{equation*}
\Big\{S(t)\in \dif u, L(t)\in \dif v, X(t)=1, N(t)=n,\sum_{j=1}^k\xi_j=k, \sum_{j=k+1}^n\xi_j=0, \xi_{n+1}=1\Big\}
\end{equation*}
can be rewritten as
\begin{align*}
\Big\{ &S(t)\in \dif u, L(t)\in \dif v, \; \sum_{j=1}^{n+1}U_j+\sum_{j=1}^k S_j+ \sum_{j=k+1}^n L_j\leq t, \\
       &\sum_{j=1}^{n+1}U_j+\sum_{j=1}^k S_j+ \sum_{j=k+1}^n
         L_j+S_{n+1}> t,
         \; \sum_{j=1}^k\xi_j=k,
         \; \sum_{j=k+1}^n\xi_j=0, \xi_{n+1}=1\Big\}.
\end{align*}
Finally, taking into account that in the case when $\{X(t)=1\}$ occupation time $S(t)=t-\sum_{j=1}^{n+1}U_j- \sum_{j=k+1}^n L_j$ and occupation time $L(t)=\sum_{j=k+1}^n L_j$, we get that
\begin{align*}
  & \phantom{=\;} p_{SL1}(u, v, t) \dif u \dif v \\
  &=\sum_{n=0}^\infty\sum_{k=0}^{n-1}
    \binom{n}{k}p_1^{k+1}p_2^{n-k}
    \Pr\left( \sum_{j=1}^{n+1}U_j \in t- \dif u-\dif v,\sum_{j=k+1}^n
                L_j \in \dif v, \sum_{j=1}^k S_j\leq u,\sum_{j=1}^k S_j+S_{n+1}>u \right)\\
  &=\sum_{n=0}^\infty\sum_{k=0}^{n-1} \binom{n}{k}p_1^{k+1}p_2^{n-k}f_U^{(n+1)}(t-u-v)f_L^{(n-k)}(v)\left[F_S^{(k)}(u)-F_S^{(k+1)}(u)\right]\dif u \dif v.
\end{align*}
\end{proof}

One can verify now that, for instance,
\begin{equation*}
p_{S1}(u,t)= \sum_{n=0}^\infty p_1^{n+1}f_U^{(n+1)}(t-u)\left[F_S^{(n)}(u)-F_S^{(n+1)}(u)\right]+\int_0^{t - u} p_{SL1}(u,v,t)\dif v.
\end{equation*}
Using symmetry between $S(t)$ and $L(t)$, we also can get that
\begin{equation*}
p_{SL2}(u, v, t)=\sum_{n=0}^\infty\sum_{k=0}^{n-1} \binom{n}{k}p_2^{k+1}p_1^{n-k}f_U^{(n+1)}(t-u-v)f_S^{(n-k)}(u)\left[F_L^{(k)}(v)-F_L^{(k+1)}(v)\right].
\end{equation*}

\section{Special Case: L\'evy Distribution}
In this section we consider a special case when all the holding times have the L\'evy
distribution (with location parameter~0 and possibly different scale parameters).
The L\'evy distribution with scale parameter $c^2$ has pdf
\begin{equation*}
g_c(x)=\frac{c}{\sqrt{2\pi}}\frac{e^{-\frac{c^2}{2x}}}{x^{3/2}},\quad x>0,
\end{equation*}
and cdf
\begin{equation*}
G_c(x)=\frac{2}{\sqrt\pi} \int_{\frac{c}{\sqrt{2x}}}^\infty e^{-t^2}\dif t, \quad x>0.
\end{equation*}
The L\'evy distribution is heavy-tailed with infinite expectation. The median is given by $0.5 c^2 \left( \mathtt{erfc}^{-1}(0.5) \right)^{-2} \approx 2.198112 c^2$, where
$\mathtt{erfc}$ is the complementary error function:
\begin{equation*}
\mathtt{erfc}(x)=\frac{2}{\sqrt\pi} \int_x^\infty e^{-t^2}\, \dif t.
\end{equation*}

\begin{figure}[tbp]
\centering
\includegraphics[width=\textwidth]{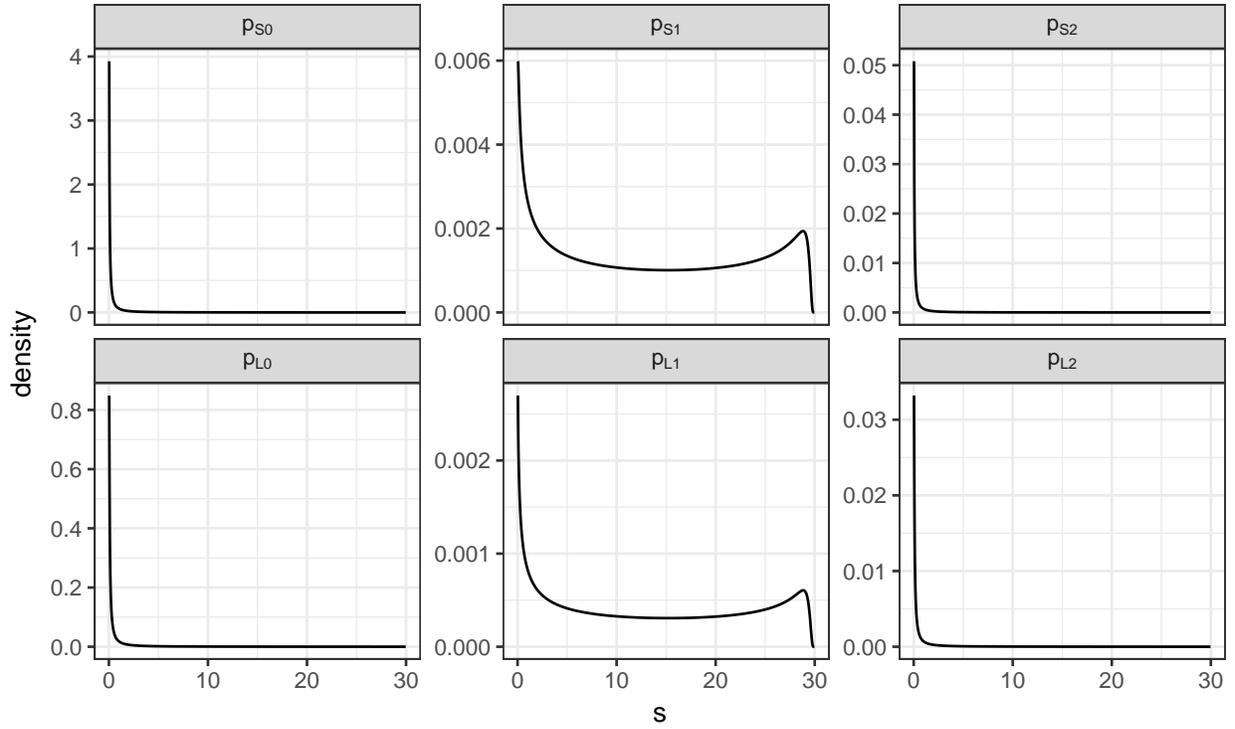}
\caption{Defective densities $p_{Sj}(s,t)$ and $p_{Lj}(s,t)$, $j \in \{0, 1, 2\}$, with $t = 30$, $c_U \approx 1.785$, $c_S \approx 0.097$, $c_L \approx 0.275$, and $p_1 = 0.9$.}
\label{fig:density}
\end{figure}

\begin{figure}[tbp]
\centering
\includegraphics[width=\textwidth]{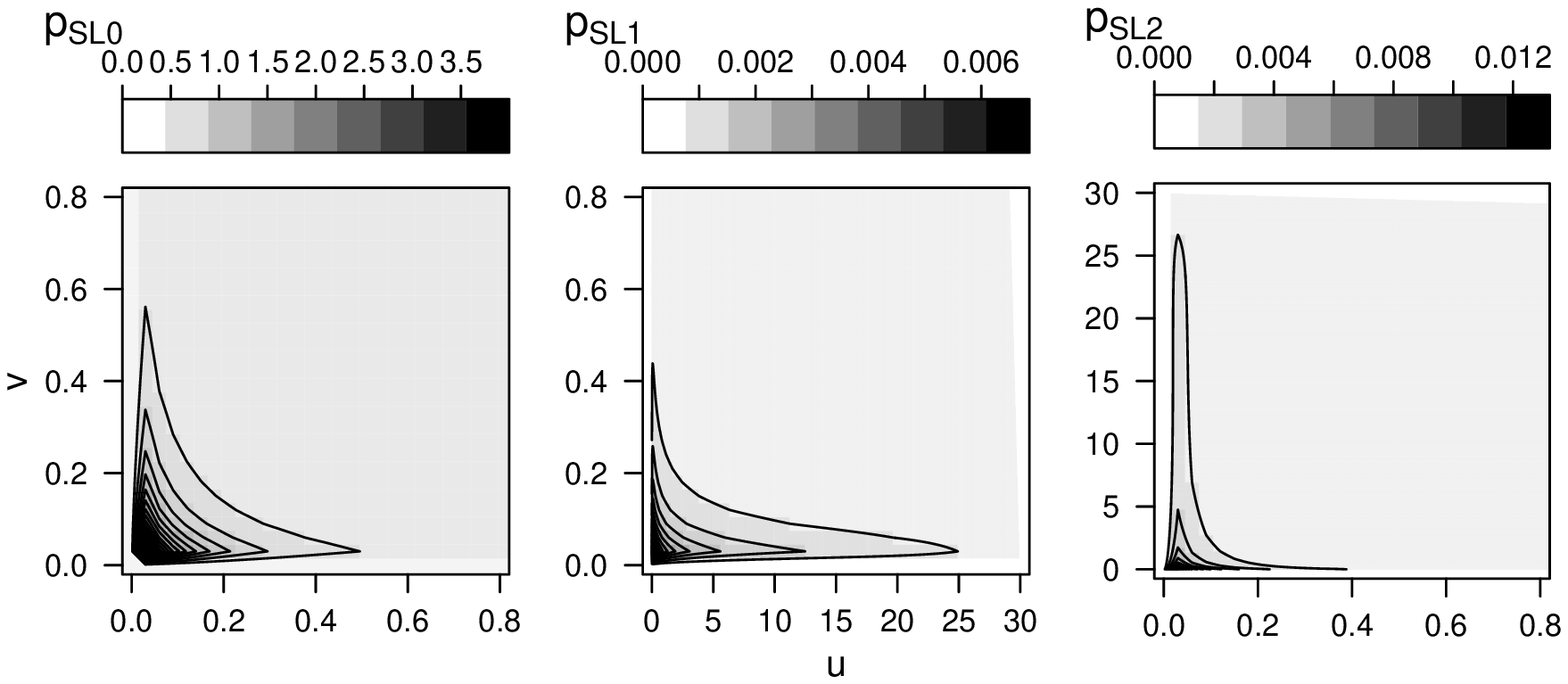}
\caption{Contour plots of the joint densities $p_{SLj}(u,v,t)$, $j \in  \{0, 1, 2\}$, with $t = 30$, $c_U \approx 1.785$, $c_S \approx 0.097$, $c_L \approx 0.275$, and $p_1 = 0.9$.}
\label{fig:joint}
\end{figure}

The parametrization that we use is a bit unusual, but it allows us to
shorten our notation for convolutions. More specifically,  as a member
of the family of stable distributions with mobility parameter~1/2, the
L\'evy distribution is closed under the operation of convolution in
the following way:
\begin{equation*}
G_{c_1}*G_{c_2}(x)=G_{c_1+c_2}(x).
\end{equation*}

Let $c_U^2$, $c_S^2$, and $c_L^2$ be the scale parameters for states
0, 1 and 2, respectively.
Then the formulas from Theorem~\ref{occupation*time} and
Theorem~\ref{joint*occupation*time}
{\it do not involve} any convolutions.
For instance, in this case we have that
\begin{equation*}
  p_{S2}(s, t) =\sum_{n=1}^\infty\sum_{k=1}^n\binom{n}{k}p_1^kp_2^{n-k+1}g_{kc_S}(s)\left[G_{(n+1)c_U+(n-k)c_L}\left(t-s\right)- G_{(n+1)c_U+(n-k+1)c_L}\left(t-s\right)\right].
\end{equation*}
The expectations and variances  of the off times is given by
\begin{align*}
  \mathrm{E}(S(t)) &= \sum_{j = 0}^2 \int_0^t  sp_{Sj}(s,t) \dif s,\quad \mathrm{Var}(S(t))= \sum_{j = 0}^2 \int_0^t  s^2p_{Sj}(s,t) \dif s - [\mathrm{E}(S(t))]^2\\
  \mathrm{E}(L(t)) &= \sum_{j = 0}^2\int_0^t  sp_{Lj}(s,t) \dif s, \quad \mathrm{Var}(L(t))= \sum_{j = 0}^2 \int_0^t  s^2p_{Lj}(s,t) \dif s - [\mathrm{E}(L(t))]^2.
\end{align*}
Note that the discrete component of the occupation times is not used
for these calculations, because the atoms are at~0.
The covariance of $S(t)$ and $L(t)$ is given by
\begin{equation*}
\mathrm{Cov}(S(t),L(t)) = \sum_{j=0}^2\iint_{u,v>0,\, u+v<t} uv p_{SLj}(u,v,t)\dif u \dif v -\mathrm{E}(S(t))\mathrm{E}(L(t)).
\end{equation*}

As an example we consider a server with median holding time 7~days, 0.5~hour,
and 4~hours for the on-state, the short off-state, and the long off-state,
respectively.  Using one day as the time unit,
the model parameters are: $c_U \approx 1.785$, $c_S \approx 0.097$,
and $c_L \approx 0.275$.
Assume also that the total repair cost is given by
\begin{equation*}
C(t)=S(t)+2L(t),
\end{equation*}
that is, long repairs are twice costlier than the short ones.

Figure~\ref{fig:density} presents defective marginal densities
$p_{Sj}(s,t)$ and $p_{Lj}(s,t)$ when $t = 30$ days and $p_1=.9$ (that is, the
less serious breakdowns occur 9~times more often).
The marginal densities for both occupation times are severely defected
when $X(30) = 0$ (the process is in the on-state).
As we mentioned above  $S(30)$ and $L(30)$ have atoms at $s=0$:
\begin{align*}
  \Pr(S(30)=0,X(30)=0)&\approx 0.280;    &   \Pr(L(30)=0,X(30)=0) &\approx 0.798;\\
  \Pr(S(30)=0,X(30)=1)&\approx 0.000;    &   \Pr(L(30)=0,X(30)=1) &\approx 0.034;\\
  \Pr(S(30)=0,X(30)=2)&\approx 0.004;    &   \Pr(L(30)=0,X(30)=2) &\approx 0.000.
\end{align*}

Figure~\ref{fig:joint} shows the contour plot of defective joint densities
$p_{SLj}(u,v,t)$ with the same parameter setup, where the negative
association of the two occupation times is obvious.
We also ran a simulation study to numerically confirm the correctness of
Theorem~\ref{occupation*time} and Theorem~\ref{joint*occupation*time}.
A total $1,000,000$ realizations of the above
on-off process were generated and the empirical results are consistent
with the theoretical densities (not shown).

Table~\ref{tab:expc} summarizes the expectation and variance of the
two types of off-state occupation times and the associated repair
cost for $t \in \{30, 60\}$ days and
$p_1 \in \{0.70, 0.75, \dots, 0.95, 0.99\}$.
As $p_1$ increases, the expectation and the
variance goes up for the less serious breakdown time but goes down fro
the more serious breakdown time. The total repair cost has lower
expectation and variance for smaller values of $p_1$.
The correlation of the two off-state occupation times is negative,
with a magnitude decreasing as $p_1$ increases.
When $t$ is doubled, the expectation of the occupation times and the
total repair cost is slightly more than doubled, which is because the
process always starts from the on-state with a random holding time.

\begin{table}[tbp]
  \centering
  \caption{Summaries of off-state occupation times and total repair
    cost with $c_U \approx 1.785$, $c_S \approx 0.097$, and
    $c_L \approx 0.275$ as $p_1$ increases.}
  \label{tab:expc}
  \addtolength\tabcolsep{-2pt}
\begin{tabular}{rrrrrrrrr}
  \toprule
  $t$ & $p_1$ & $\mathrm{E}(S(t))$ & $\mathrm{Var}(S(t))$ & $\mathrm{E}(L(t))$ & $\mathrm{Var}(L(t))$ & $\mathrm{E}(C(t))$ & $\mathrm{Var}(C(t))$ & $\rho (S(t), L(t))$ \\
  \midrule
  30 & 0.70 & 0.81 & 10.82 & 0.95 & 13.01 & 2.71 & 60.99 & $-$0.040 \\
   & 0.75 & 0.87 & 11.60 & 0.79 & 10.95 & 2.46 & 53.71 & $-$0.037 \\
   & 0.80 & 0.93 & 12.38 & 0.64 & 8.85 & 2.20 & 46.32 & $-$0.035 \\
   & 0.85 & 0.99 & 13.15 & 0.48 & 6.70 & 1.95 & 38.80 & $-$0.031 \\
   & 0.90 & 1.05 & 13.93 & 0.32 & 4.51 & 1.69 & 31.16 & $-$0.026 \\
   & 0.95 & 1.11 & 14.70 & 0.16 & 2.28 & 1.43 & 23.38 & $-$0.019 \\
   & 0.99 & 1.16 & 15.32 & 0.03 & 0.46 & 1.23 & 17.07 & $-$0.009 \\
  60 & 0.70 & 1.75 & 48.20 & 2.08 & 57.96 & 5.92 & 271.58 & $-$0.040 \\
   & 0.75 & 1.88 & 51.68 & 1.74 & 48.81 & 5.37 & 239.32 & $-$0.038 \\
   & 0.80 & 2.02 & 55.16 & 1.40 & 39.46 & 4.82 & 206.48 & $-$0.035 \\
   & 0.85 & 2.15 & 58.64 & 1.05 & 29.91 & 4.26 & 173.04 & $-$0.031 \\
   & 0.90 & 2.29 & 62.13 & 0.70 & 20.16 & 3.70 & 139.01 & $-$0.026 \\
   & 0.95 & 2.42 & 65.62 & 0.35 & 10.19 & 3.13 & 104.37 & $-$0.019 \\
   & 0.99 & 2.53 & 68.41 & 0.07 & 2.05 & 2.67 & 76.21 & $-$0.009 \\

 \bottomrule
\end{tabular}
\end{table}

\section{Concluding Remarks}

The obtained results can be generalized to the case of more than two
off-states. The only difference is that instead of binomial
probabilities one has to use multinomial ones.

The marginal distribution of an off-state occupation time could be
obtained by collapsing the on-state with the other off-state and using
results on the telegraph process.
Nonetheless, this approach would lead to more complicated
formulas because the holding time of the new collapsed state is
distributed as an infinite mixture of convolutions. And of course,
this technique cannot be employed if we need the joint distribution
of the two off-state occupation times.

The process $C(t)$ also can be viewed as a telegraph process governed
by the on-off process with multiple off-states. That is,
$C(t)$ is the time $t$ position of a particle that moves with speed
$C_j$ whenever $X(t)=j$ ($j=0,1,2$). The formulas for the joint
distribution of $X(t)$ and occupation times can be useful for
parameter estimation when process $C(t)$ is discretely observed.
In particular, if $X(t)$ is also Markov (that is, all the holding
times are exponential) then the efficient likelihood
estimation is possible with help of tools for hidden Markov model
\citep{Pozd:etal:2018}.

\bibliographystyle{mcap}
\bibliography{OTTP}

\begin{thebibliography}{19}
\newcommand{\enquote}[1]{``#1''}
\expandafter\ifx\csname natexlab\endcsname\relax\def\natexlab#1{#1}\fi

\bibitem[{Bshouty et~al.(2012)Bshouty, Di~Crescenzo, Martinucci and
  Zacks}]{Bshouty:etal:2012}
D.~Bshouty, A.~Di~Crescenzo, B.~Martinucci and S.~Zacks (2012).
\newblock \enquote{Generalized telegraph process with random delays.}
\newblock {\em Journal of Applied Probability\/} {\bf 49}, 850--865.

\bibitem[{Cane(1959)}]{Cane:beha:1959}
V.~R. Cane (1959).
\newblock \enquote{Behavior sequences as semi-{M}arkov chains.}
\newblock {\em Journal of the Royal Statistical Society, Series B\/} {\bf 21},
  36--58.

\bibitem[{Crimaldi et~al.(2013)Crimaldi, Di~Crescenzo, Iuliano and
  Martinucci}]{Crimaldi:etal:2013}
I.~Crimaldi, A.~Di~Crescenzo, A.~Iuliano and B.~Martinucci (2013).
\newblock \enquote{A generalized telegraph process with velocity driven by
  random trials.}
\newblock {\em Advances in Applied Probability\/} {\bf 45}, 1111--1136.

\bibitem[{De~Gregorio and Macci(2012)}]{DeGregorio:etal:2012}
A.~De~Gregorio and C.~Macci (2012).
\newblock \enquote{Large deviation principles for telegraph processes.}
\newblock {\em Statistics \& Probability Letters\/} {\bf 82}, 1874--1882.

\bibitem[{Di~Crescenzo(2001)}]{DiCrescenzo:2001}
A.~Di~Crescenzo (2001).
\newblock \enquote{On random motions with velocities alternating at
  {E}rlang-distributed random times.}
\newblock {\em Advances in Applied Probability\/} {\bf 33}, 690--701.

\bibitem[{Di~Crescenzo et~al.(2013)Di~Crescenzo, Iuliano, Martinucci and
  Zacks}]{DiCrescenzo:etal:2013}
A.~Di~Crescenzo, A.~Iuliano, B.~Martinucci and S.~Zacks (2013).
\newblock \enquote{Generalized telegraph process with random jumps.}
\newblock {\em Journal of Applied Probability\/} {\bf 50}, 450--463.

\bibitem[{Di~Crescenzo et~al.(2014)Di~Crescenzo, Martinucci and
  Zacks}]{DiCrescenzo:etal:2014}
A.~Di~Crescenzo, B.~Martinucci and S.~Zacks (2014).
\newblock \enquote{On the geometric brownian motion with alternating trend.}
\newblock {\em Mathematical and statistical methods for actuarial sciences and
  finance, Eds. Perna, C. and Sibillo, M.\/} pp. 81--85.

\bibitem[{Di~Crescenzo and Zacks(2015)}]{DiCrescenzo:Zhacs:2015}
A.~Di~Crescenzo and S.~Zacks (2015).
\newblock \enquote{Probability law and flow function of {B}rownian motion
  driven by a generalized telegraph process.}
\newblock {\em Methodology and Computing in Applied Probability\/} {\bf 17},
  761--780.

\bibitem[{Kolesnik and Ratanov(2013)}]{Kolesnik:Ratanov:2013}
A.~D. Kolesnik and N.~Ratanov (2013).
\newblock {\em Telegraph processes and option pricing\/}.
\newblock Springer Briefs in Statistics. Springer, Heidelberg.

\bibitem[{Macci(2016)}]{Macci:2016}
C.~Macci (2016).
\newblock \enquote{Large deviations for some non-standard telegraph processes.}
\newblock {\em Statistics \& Probability Letters\/} {\bf 110}, 119--127.

\bibitem[{Newman(1968)}]{Newman:1968}
D.~S. Newman (1968).
\newblock \enquote{On the probability distribution of a filtered random
  telegraph signal.}
\newblock {\em The Annals of Mathematical Statistics\/} {\bf 39}, 890--896.

\bibitem[{Page(1960)}]{Page:theo:1960}
E.~S. Page (1960).
\newblock \enquote{Theoretical considerations of routine maintenance.}
\newblock {\em The Computer Journal\/} {\bf 2}, 199--204.

\bibitem[{Perry et~al.(1999)Perry, Stadje and Zacks}]{Perr:Stad:Zack:firs:1999}
D.~Perry, W.~Stadje and S.~Zacks (1999).
\newblock \enquote{First-exit times for increasing compound processes.}
\newblock {\em Communications in Statistics: Stochastic Models\/} {\bf 15},
  977--992.

\bibitem[{Pozdnyakov et~al.(2018)Pozdnyakov, Elbroch, Hu, Meyer and
  Yan}]{Pozd:etal:2018}
V.~Pozdnyakov, L.~Elbroch, C.~Hu, T.~Meyer and J.~Yan (2018).
\newblock \enquote{On estimation for {B}rownian motion governed by telegraph
  process with multiple off states.}
\newblock {\em ArXiv e-prints\/} .

\bibitem[{Pozdnyakov et~al.(2017)Pozdnyakov, Elbroch, Labarga, Meyer and
  Yan}]{Pozd:etal:2017}
V.~Pozdnyakov, L.~Elbroch, A.~Labarga, T.~Meyer and J.~Yan (2017).
\newblock \enquote{Discretely observed {B}rownian motion governed by telegraph
  process: estimation.}
\newblock {\em Methodology and Computing in Applied Probability\/} {t}o appear.

\bibitem[{Ratanov(2017)}]{Ratanov:2017}
N.~Ratanov (2017).
\newblock \enquote{Piecewise linear process with renewal starting points.}
\newblock {\em Statistics \& Probability Letters\/} {\bf 131}, 78--86.

\bibitem[{Stadje and Zacks(2004)}]{Stad:Zack:tele:2004}
W.~Stadje and S.~Zacks (2004).
\newblock \enquote{Telegraph processes with random velocities.}
\newblock {\em Journal of Applied Probability\/} {\bf 41}, 665--678.

\bibitem[{Xu et~al.(2015)Xu, De and Zacks}]{Xu:etal:2015}
Y.~Xu, S.~K. De and S.~Zacks (2015).
\newblock \enquote{Exact distribution of intermittently changing positive and
  negative compound {P}oisson process driven by an alternating renewal process
  and related functions.}
\newblock {\em Probability in the Engineering and Informational Sciences\/}
  {\bf 29}, 385--397.

\bibitem[{Zacks(2004)}]{Zack:gene:2004}
S.~Zacks (2004).
\newblock \enquote{Generalized integrated telegraph processes and the
  distribution of related stopping times.}
\newblock {\em Journal of Applied Probability\/} {\bf 41}, 497--507.

\end{thebibliography}

\end{document}